\documentclass[10pt,psamsfonts]{amsart}
\usepackage{amsmath}
\usepackage{amsthm}
\usepackage{amssymb}
\usepackage{multicol}
\usepackage{amscd}
\usepackage{amsfonts}
\usepackage{amsbsy}
\usepackage{epsfig,afterpage}
\usepackage[dvips]{psfrag}
\usepackage{epsfig}
\usepackage{amsmath}
\usepackage[margin=3.6cm]{geometry}
\usepackage{srcltx}
\usepackage{amsfonts}
\usepackage{amssymb}
\usepackage{psfrag}
\usepackage{multicol}
\usepackage{verbatim}
\usepackage{graphicx}
\usepackage{amsmath}
\usepackage{amsthm}
\usepackage{amssymb}
\usepackage{amscd}
\usepackage{amsfonts}
\usepackage{amsbsy}
\usepackage{epsfig}
\usepackage[dvips]{psfrag}
\usepackage{multirow}
\usepackage{graphics}
\usepackage{epstopdf}
\usepackage{overpic}
\usepackage{float}
\usepackage[utf8]{inputenc}

\usepackage{graphicx}

\usepackage[colorlinks=true, linkcolor=blue, citecolor=red]{hyperref}

\newtheorem {theorem} {Theorem}
\newtheorem {proposition} [theorem]{Proposition}
\newtheorem {corollary} [theorem]{Corollary}
\newtheorem {lemma}  [theorem]{Lemma}

\newtheorem {remark} [theorem]{Remark}

\newtheorem{mtheorem}{Theorem}

\usepackage{float}
\usepackage{tikz, wrapfig}
\tikzset{node distance=3cm, auto}
\allowdisplaybreaks

\begin{document}

\title[Simultaneous bifurcation of limit cycles for PWHS]
{Simultaneous bifurcation of limit cycles for Piecewise Holomorphic systems}

\author[Armengol Gasull, Gabriel Rondón and Paulo R. da Silva]{Armengol Gasull$^{1}$, Gabriel Rondón$^{1}$ and Paulo R. da Silva$^{2}$}
\address{$^{1}$Departament de Matemàtiques, Edifici Cc, Universitat Autònoma de Barcelona, 08193 Bellaterra, Barcelona, Spain.}
\address{$^{2}$S\~{a}o Paulo State University (Unesp), Institute of Biosciences, Humanities and
	Exact Sciences. Rua C. Colombo, 2265, CEP 15054--000. S. J. Rio Preto, S\~ao Paulo,
	Brazil.}
 
\email{armengol.gasull@uab.cat}
\email{garv202020@gmail.com}
\email{paulo.r.silva@unesp.br}



\date{}
\maketitle

\begin{abstract}
Let $\dot{z}=f(z)$ be a holomorphic differential equation with center at $p$.  In this paper we are concerned about studying the piecewise perturbation systems $\dot{z}=f(z)+\epsilon R^\pm(z,\overline{z}),$ where $R^\pm(z,\overline{z})$ are complex polynomials defined for $\pm\operatorname{Im}(z)> 0.$  We provide an integral expression, similar to an Abelian integral, for the period annulus of  $p.$ The zeros of this integral control the bifurcating limit cycles from the periodic  orbits of this annular region. This expression is given in terms of the conformal conjugation between $\dot{z}=f(z)$ and its linearization $\dot{z}=f'(p)z$ at $p$. We use this result to control the simultaneous bifurcation of limit cycles of the two annular periods of $\dot{z}={\rm i} (z^2-1)/2$, after both complex and holomorphic piecewise polynomial perturbations.  In particular, as far as we know,  we provide the first proof of the existence of non nested limit cycles for piecewise holomorphic systems.
\end{abstract}


\section{Introduction}

There are several important aspects to understand the dynamics of planar differential systems such as knowledge of the existence and number of limit cycles. In fact, the famous Hilbert's 16th problem is one of the main open problems in the qualitative theory of planar polynomial vector fields. Finding good upper or lower bounds for the maximum number of limit cycles of particular families of such systems in terms of their degrees, or other characteristics, constitute challenging  problems. Of course, the existence of upper bounds, explicit or not, are the more difficult questions. In this work we will concentrate on lower bounds.

In recent years, great interest has arisen in the study of limit cycles of piecewise holomorphic systems, which is a subfamily of piecewise smooth systems. This is because holomorphic functions have many applications in various areas of applied science, such as the study of fluid dynamics \cite{BatGK,Mars,Conw}. Furthermore, the study and properties of holomorphic systems $\dot{z}=f(z)$ make them interesting and beautiful but the absence of limit cycles makes them dynamically poor. Interestingly, in \cite{GASRONSIL1} and \cite{Rondon2022440} the authors showed that there are piecewise holomorphic systems that have limit cycles. More precisely, in \cite{Rondon2022440} the authors have used the intrinsic properties of holomorphic functions, such as their integrability, to construct limit cycles, whereas in \cite{GASRONSIL1} Gasull et al. approach this problem with different points of view: study of the number of zeros of the first and second order averaged functions, and with the control of the limit cycles appearing from a monodromic equilibrium point via a degenerated Andronov--Hopf type bifurcation.

Consider the piecewise polynomial complex systems (PWCS), 
\begin{equation}
\begin{aligned}\label{eq_pert}
\dot{z}=f(z)+\left\{\begin{array}{l}
\epsilon R^+_m(z,\overline{z}), \text{ when } \operatorname{Im}(z)> 0,\\[5pt]
\epsilon R^-_m(z,\overline{z}),\text{ when } \operatorname{Im}(z)<0,
\end{array} \right.
\end{aligned}
\end{equation}
where $\dot{z}=f(z)$ has a center at $p,$ $0<\epsilon\ll1,$ $z=x+iy\in\mathbb{C}$ and $R_m^{\pm}(z,\overline{z})$ are complex polynomial functions with  degree $m.$  Notice that
the straight line  $\Sigma=\{z\in\mathbb{C}:\operatorname{Im}(z)=0\}$ divides the plane in two halfplanes  $\Sigma^\pm=\{z\in\mathbb{C}:\pm\operatorname{Im}(z)>0\}.$ As usual, the orbits on $\Sigma$ are defined following the Filippov convention, see \cite{Filippov88}  for more details. 

When in system~\eqref{eq_pert}, $f(z)= {\rm i}z,$  in the  $(r,\theta)-$polar coordinates $z=re^{i\theta}$, $r>0$ and    $\theta\in \mathbb{S}^1,$  it is converted into
\begin{equation}\label{polar_sys_int}
\dfrac{dr}{d\theta}=\begin{aligned}
\left\{\begin{array}{l}
F^+(\theta,r,\epsilon)=\epsilon F_1^+(\theta,r)+\mathcal{O}(\epsilon^{2}), \text{ if } 0\leq\theta\leq\pi,\\[5pt]
F^-(\theta,r,\epsilon)=\epsilon F_1^-(\theta,r)+\mathcal{O}(\epsilon^{2}), \text{ if } \pi\leq\theta\leq2\pi,
\end{array} \right.
\end{aligned}
\end{equation}
where $\epsilon>0$ is a sufficiently small parameter and $\mathcal{O}$ represents terms of order at least two in~$\varepsilon$ for the functions~$F^\pm.$  By using the theory of averaging in this context (see~\cite{MR3729598}),  it is well-known that each simple zero $r=r_0$ of  
\begin{equation}\label{abeliana}
    M_1(r)=M_1^+(r)-M_1^-(r)\quad\mbox{where}\quad M_1^\pm(r)=\displaystyle\int_0^{\pm\pi} F_1^\pm(\theta,r) d\theta,
\end{equation}
provides, for $\epsilon$ small enough, a hyperbolic limit cycle of the piecewise smooth system \eqref{polar_sys_int} that tends to $r=r_0$ when $\epsilon$ tends to 0. The function $M_1$ is called the {\it first order averaged function} and sometimes it is also called the {\it first order Melnikov function}.

Our first aim is to use expression~\eqref{abeliana} to obtain a general closed expression for the Melnikov function of system~\eqref{eq_pert} for a general $f$ such that $\dot z=f(z)$ has a center at $p.$ From \cite{BLT,GGJ}, we know that there exists a conformal map $w=\phi(z)$ such that this differential equation can be written as $\dot{w}= -{\rm i}w$ (or $\dot{w}= {\rm i}w$). The map $\phi$ is called the  {\it linearizing change} of $\dot{z}=f(z)$ at~$p$. Our approach works on the largest open set where this conformal map is well defined. In what follows, we state our first main result.

\begin{mtheorem}\label{tma}
    Consider the piecewise complex system \eqref{eq_pert}. Suppose that $\phi$ is the linearizing change of  $\dot{z}=f(z)$ at $p$ such that $\phi(\Sigma)\subset\Sigma$. Then, its first order Melnikov function is
$M_1(r)=M_1^+(r)-M_1^-(r),$ where
\begin{equation}
    M_1^\pm(r)=  -\operatorname{Im}\left(\displaystyle\int_{0}^{\pm\pi}\overline{\phi'(\phi^{-1}(re^{{\rm i}\theta}))R_m^\pm\left(\phi^{-1}(re^{{\rm i}\theta}),\overline{\phi^{-1}(re^{{\rm i}\theta}})\right)}\,{\rm i}e^{{\rm i}\theta}d\theta\right).
\end{equation}
In particular,  each simple zero $r=r_0$ of $M_1$ provides, for $\epsilon$ sufficiently small, a hyperbolic limit cycle of \eqref{eq_pert} that tends to $r=r_0$ when $\epsilon$ tends to 0.
\end{mtheorem}

The above results is an extension to the discontinuous case to the one obtained in~\cite{GARIJO2008813} in the smooth situation.

We will employ Theorem \ref{tma} to study PWCS \eqref{eq_pert} for $f(z)={\rm i}(z^2-1)/2$ and 
\begin{equation}\label{tom}
R_m^\pm(z,\overline{z})=\sum_{l=0}^m\sum_{k=0}^{l}\overline{a}^\pm_{k,l}z^{l-k}\overline{z}^k,\,\, a_{k,l}\in\mathbb{C}\,\,\text{and}\,\, m=0,1,2,3.
\end{equation}
In fact, this unperturbed system has been also the one considered in~\cite{GARIJO2008813} in the smooth perturbations context.

We emphasize that $\dot z=f(z)={\rm i}(z^2-1)/2$ has 2 centers at -1 and 1,  separated by the invariant straight line $\operatorname{Re}(z)=0.$ Each of the punctured halfplanes $\{z\in\mathbb{C}:\pm\operatorname{Re}(z)>0\} \setminus\{\pm 1\}$ is filled by periodic orbits of the system.  To carry out this study, we explicitly use the linearization change $w=\phi(z)=(1+z)/(1-z)$ of $\dot{z}=f(z)$ at $-1$, which we employ to find the bifurcation function at $z=-1$. Thus, by changing variables and time we will obtain the bifurcation function at $z=1.$ It is worth noting that this linearizing  is specially simple and has also a simple inverse but, unfortunately, for most holomorphic vector fields $f$, the calculations can be much more complicated.

This type of problem has been addressed in several papers. Specifically, in \cite{GARIJO2008813} Garijo et al. study the smooth case, that is, $R_m=R^+_m=R^-_m$, providing an integral expression for the differential equation $\dot{z}=f(z)+\epsilon R_m( z,\overline{z})$ and use this formula to control the simultaneous bifurcation of limit cycles of the two annular periods of $\dot{z}= {\rm i}z + z^2$, 
after a polynomial perturbation. Also, in \cite{DACRUZ2018248} the authors investigate the number of bifurcating periodic orbits of a cubic polynomial vector field having two period rings using piecewise perturbations. They study, up to first-order averaging analysis, the bifurcation of periodic orbits of the two annular periods, the first separately and the second simultaneously. There are other works that consider the problem of simultaneous bifurcation such as \cite{CHICONE1991268,article_Gasull,Zhao2002} although in the context of piecewise systems the type of bifurcation that we consider in this paper is a complete novelty.

Before stating our second main result, we introduce some notations. We denote by $M_1$ and~$N_1$ the first order averaged functions  at $-1$ and $1$, respectively. In addition, we say that system~\eqref{eq_pert} presents the {\it configuration of limit cycles $[i,j]$}  if  $M_1$ and $N_1$ have simultaneously~$i$ and $j$ simple zeros in the interval $(0,1),$ respectively. Our definition is motivated from the theory of  averaging of first order, because as we will see in this case, for $\varepsilon$ small enough, the differential system  has $i$ limit cycles surrounding~$-1$ and $j$~limit cycles surrounding~$1.$ Obviously, by the symmetry of the problem if the configuration $[i,j]$ holds the configuration $[j,i]$ also does. For short we will say that the {\it configuration $[[i,j]]$ is realizable.}

For  $m\le 3,$ we will prove that  both functions, multiplied by $r,$ belong to the vectorial space $\mathcal{F}$ generated by the functions 
 $$\mathcal{F}=[r,r^2,r^3,r^4,(r^2-1)^2\operatorname{arctanh}(r),(r^4-1)\operatorname{arctanh}(r),r^2\operatorname{arctanh}(r)],$$
where recall that  $$\operatorname{arctanh}(r)=\tanh^{-1}(r)=\frac12\ln\left(\frac{1+r}{1-r}\right).$$
As we will prove, the above ordered set of functions forms an  {\it extended
complete Chebyshev system (ECT-system) on $(0,1),$} see Section~\ref{misc} for more details. This property will allow to control the exact number of zeros of each of the functions, separately.

For bigger $m$ the number of monomials $r^k,$ as well as the number of functions of the form $S_l(r)\operatorname{arctanh}(r)$ for some fixed increasing degree polynomials $S_l$ in $M_1$ and $N_1$ will grow, but our approach also would apply. In short, for a fixed $m$ the control of the maximum number of zeros of $M_1$ and $N_1$ seems that could be completely understood, but would need much more computational effort. For this reason we have restricted most of our attention to the case $m\le3.$ On the other hand, the knowledge of the maximum number that both functions can have simultaneously is a difficult and challenging problem. All results that we have obtained in this direction are resumed in next theorems.

\begin{mtheorem}\label{tmb} Consider PWCS \eqref{eq_pert} with  $f(z)={\rm i}(z^2-1)/2$. If $m=0,$ the unique realizable configurations are  $[[i,j]],$ with $0\leq i+j\leq 1$. When $1\le m\le 3,$  if $[[i,j]]$ is a realizable configuration then $i,j\le m+3.$  Moreover the following configurations, of course satisfying   $i,j\le m+3,$ are realizable:
 \begin{itemize}       
        \item[(a)] $[[i,j]]$ with $0\leq i+j\leq 4$ when $m=1.$ 
         \item[(b)] $[[i,j]]$ with $0\leq i+j\leq 6,$  when $m=2.$
          \item[(c)] $[[i,j]]$ with $0\leq i+j\leq 8$ when $m=3.$
            \end{itemize}
\end{mtheorem}

As a very particular subcase in the proof of Theorem~\ref{tmb} there appears the situation when both $M_1$ and $N_1$ are polynomials. In this situation the question of the simultaneous number of zeros in $(0,1)$ of both functions can be approached with much  more detail. We believe that this is a problem that is interesting by itself. All our results in this subcase are resumed in Proposition~\ref{pro:pol}.

At this point another natural question arises: What happens if the complex perturbation function $R^\pm_m$ in  PWCS~\eqref{eq_pert} is holomorphic?  For short we will call these systems  PWHS. For them both $R^\pm_m$ depend only on $z$. As we will see, the fact that the perturbation is holomorphic greatly simplifies the calculations and it is not surprising that the number of cycles that arise is less than in the above more general situation. 

As we will prove, in this case and for   $m\ge 3,$ the functions  $M_1(r)$ and $N_1(r),$ multiplied by $(r^2-1)^{m-3},$ belong to the vectorial space $\mathcal{G}$ generated by the functions 
$$\mathcal{G}=[1,r,r^2,\ldots, r^{2(m-2)}, r(r^2-1)^{m-3}\operatorname{arctanh}(r)],$$
that is also an ECT-system on $(0,1).$  We only have tackled the problem of simultaneous bifurcations when $m\le 3.$ Our main result for PWHS is:

\begin{mtheorem}\label{tmc}  Consider piecewise holomorphic systems of the form~\eqref{eq_pert}, with  $f(z)={\rm i}(z^2-1)/2,$ where $R^\pm_m$ depend only on $z$. The following holds:
	
(a) When $m=0$ the unique realizable configurations are  $[[i,j]],$ with $0\leq i+j\leq 1$ and
 when $m\in\{1,2\}$ the unique realizable configurations are  $[[i,j]]$ with $i+j\le 2.$

(b)   When $m=3,$  if $[[i,j]]$ is a realizable configuration then $i,j\le 3.$   
Moreover the following configurations, of course satisfying   $i,j\le 3,$ are realizable: $[[i,j]]$ with $0\leq i+j\leq 4.$

 (c)  When $m> 3,$  if $[[i,j]]$ is a realizable configuration then $i,j\le 2m-3.$ 
\end{mtheorem}

 In item~(c) we only give an upper bound for the values $i$ and $j.$ We believe that this upper bound is reached, as happens when $m=3.$ To prove this fact we should develop in more detail our computations but we have decided do not face this question here. Similarly, we think that  the value $i+j$ must have an upper bound smaller that $2(2m-3),$ because the functions $M_1$ and $N_1$ are strongly dependent, see the proof of Proposition~\ref{prop_aux_hol}.

It is well known that smooth quadratic systems can have nested limit cycles, formed by $1,2$ or $3$ limit cycles and also limit cycles forming two disjoint nests with configurations $\{1,1\},$ $\{2,1\}$ and $\{3,1\},$ see~\cite{Zeg} and its references. Moreover, people believe that these are the only possible configurations. Until now, all examples with PWHS having limit cycles present them in a single nest and, already in the degree 1 case, there were  linear PWHS examples with $1,2$ or $3$ nested limit cycles, see~\cite{GASRONSIL1}. In this paper we present the first examples of PWHS with two different nests of limit cycles. Moreover for quadratic PWHS we obtain the following result:

\begin{corollary}
There are quadratic PWHS of the form \eqref{eq_pert}, with $R^\pm_m$ depending only on $z$ and $m\in\{1,2\},$ having two limit cycles with configuration $[[1,1]].$ For cubic PWHS of the same form but with $m=3,$ there are configurations with two nests of the types $[[1,1]],$ $[[2,1]],$ $[[3,1]]$ and $[[2,2]].$
\end{corollary}

The paper is organized as follows. In Section \ref{sec:Preliminaries} we present some basic results and the proof  of Theorem~\ref{tma}.  Then we dedicate next two sections to prove Theorems~\ref{tmb}, and~\ref{tmc}.  The most tedious  computations, devoted to obtain the first order averaged functions $M_1$ and $N_1,$ are deferred to the Appendix~\ref{app}. 

\section{Preliminaries and  proof  of Theorems~\ref{tma}}\label{sec:Preliminaries}
In this section first we recall some results that will be used throughout the paper. Then we  prove Theorem~\ref{tma} and by using it and the computations of Appendix~\ref{app} we obtain the explicit expressions of first order averaged functions $M_1$ and $N_1.$

\subsection{The averaging method}\label{aver_metd}

We briefly recall some basic results of the averaging theory for piecewise smooth systems written in polar coordinates. An overview on this subject can be found in \cite{MR3729598}, and the reader can see the details of the proofs there. Consider the piecewise smooth systems of the form 
\begin{equation}\label{polar_sys}
\dfrac{dr}{d\theta}=\begin{aligned}
\left\{\begin{array}{l}
F^+(\theta,r,\epsilon) \text{ if } 0\leq\theta\leq\pi,\\[5pt]
F^-(\theta,r,\epsilon) \text{ if } \pi\leq\theta\leq2\pi,
\end{array} \right.
\end{aligned}
\end{equation}
where $F^\pm(\theta,r,\epsilon)=\sum_{j=1}^{k}\epsilon^jF_j^\pm(\theta,r)+\epsilon^{k+1}R^\pm(\theta,r,\epsilon),$ with $\theta\in S^1,$ $r>0$ and $\epsilon>0$ is a sufficiently small parameter. 

The following result can be found in \cite[Theorem 1]{MR3729598}:
\begin{theorem}\label{prop3} Let $M_1$ be the averaged function of order~1  given by \eqref{abeliana}.  Then, each simple zero $r=r_0$ of  $M_1$ provides, for $\epsilon$ small enough, a hyperbolic limit cycle of the piecewise smooth system \eqref{polar_sys} that tends to $r=r_0$ when $\epsilon$ tends to 0.
\end{theorem}

\subsection{A miscellany of results}\label{misc}
 
A very useful and well-kown characterization of extended
complete Chebyshev system (ECT-system) is the following:
\begin{lemma}\label{wrons}
    $[f_0,\cdots, f_n]$ is an ECT-system on $I\subset\mathbb{R},$ an open interval, if and only if for all $k=0,1,\cdots,n,$ $W_k(x)\neq0$ for all $x\in I$, where
$$W_k(x)=W [f_0, . . . , f_k](x)= \det\left(f^{(i)}_j(x)\right)_{0\leq i,j\leq k}$$
is the Wronskian of $f_0,\cdots, f_k$ at $x\in I$.
\end{lemma}
This result allows us to estimate the number of real zeros of any non-zero function $F\in \operatorname{Span}\{f_0,\cdots, f_n\}$, where $\operatorname{Span}(\mathcal{F})$ denotes the set of all functions given by linear combinations of the functions of $\mathcal{F}$. In what follows, we state a classical result related to the ECT-system, whose proof can be found in \cite{Karlin}. 
\begin{theorem}\label{teo_wrons}
     Let $\mathcal{F}=[f_0, . . . , f_n]$ be an ECT-system on
$I$. Then, the number of isolated zeros for every element of $\operatorname{Span}(\mathcal{F})$ does not exceed $n$. Moreover, for each configuration
of $m\leq n$ zeros in $I,$ taking into account their multiplicity, there exists
$F\in \operatorname{Span}(\mathcal{F})$ with this configuration of zeros.
\end{theorem}

In what follows, we provide a simple result for finding the zeros of $k$-parameter families of polynomials in one variable.

Let $F_\lambda(x)$ be a $k$-parametric family of polynomials. We denote the discriminant of a polynomial $p(x) = a_nx^n+\cdots+a_1x+a_0$ as $\Delta_x(p)$, i.e.
$$\Delta_x(p) = (-1)^{\frac{n(n-1)}{2}}\frac{1}{a_n}\operatorname{Res}(p(x),p'(x)),$$
where $\operatorname{Res}(p,p')$ is the resultant of $p$ and $p_0.$

Using the same ideas as in \cite[Lemma 8.1]{Johanna}, it is easy to prove the following result, which will be used throughout the paper.

\begin{lemma}\label{roots_1pf}
Let $F_\lambda(x) = f_n(\lambda)x^n + f_{n-1}(\lambda)x^{n-1}+\cdots+ f_1(\lambda)x+f_0(\lambda)$, $n>1$, be a family of real polynomials depending continuously on a parameter $\lambda\in\mathbb{R}^k$ and set $\Omega_\lambda=(a(\lambda),b(\lambda))$, for some continuous functions $a(\lambda)$ and $b(\lambda)$. Assume that there exists an connected open set $\mathcal{U}\subset\mathbb{R}^k$ such that: 
\begin{itemize}
   \item[(i)] For some $\lambda_0\in\mathcal{U}$, $F_{\lambda_0}$ has exactly $m$ zeros in $\Omega_{\lambda_0}$ and all of them are simple.
    \item[(ii)]  For all $\lambda\in\mathcal{U}$, $F_{\lambda}(a(\lambda))\cdot F_{\lambda}(b(\lambda))\neq0$.
    \item[(iii)] For all $\lambda\in\mathcal{U}$, $\Delta_x(F_\lambda)\neq0$.
\end{itemize}
Then for all $\lambda\in \mathcal{U}$, $F_{\lambda}$ has also exactly $m$ zeros in $\Omega_{\lambda}$ and all of them are simple. 
\end{lemma}

To finish this section we state the well-known  \textit{Descartes Theorem}, which provides information about the number of positive zeros of a real polynomial based on the sign changes and the number of terms. For further details, see, for example, \cite{MR0174165}. Given an ordered list of $p+1$ non-zero real numbers $[a_0, a_1, \ldots, a_p]$, we define the number of sign variations, denoted by $m$ with $0 \le m \le p$, as the number of indices $j \le p - 1$ for which $a_j a_{j+1} < 0$. 

\begin{theorem}[Descartes Theorem]\label{descartes}
	Consider the real polynomial $P(x)=a_{0}x^{i_0}+\dots+a_{p}x^{i_p}$ with $0\leq i_0<\dots<i_p$ and  $a_{j}$ non-zero real constants for $j\in\{0,\ldots,p\}.$ If the number of sign variations  of $[a_0,a_1,\ldots,a_p]$  is $m$, then $P(x)$ has exactly   $m-2n$ positive real zeros counting their multiplicities, where $n$ is a non negative integer number.
\end{theorem}

\subsection{Proof of Theorem \ref{tma}}
Since $p$ is a center of $\dot{z}=f(z)$, then from \cite{BLT,GGJ} we know that there exists $\phi$ conformal map such that $\phi'(z)f(z)=-{\rm i}\phi(z).$ Using this on $\dot{z}=f(z)+\epsilon R^\pm(z,\overline{z}),$ we get 
\begin{equation}\label{eq_cl}
    \dot{w}=-{\rm i}w+\epsilon L^\pm(w,\overline{w}),
\end{equation}
where $L^\pm(w,\overline{w})=\phi'(\phi^{-1}(w))R^\pm_m(\phi^{-1}(w),\overline{\phi^{-1}(w)}).$ In the $(r,\theta)-$coordinates $w=re^{i\theta}$, \eqref{eq_cl} is converted into
\begin{equation}\label{polar_eq}
    \frac{dr}{d\theta}=\frac{r\epsilon\Big(\operatorname{Re}(L^\pm(re^{{\rm i}\theta},re^{-{\rm i}\theta}))\cos(\theta)+\operatorname{Im}(L^\pm(re^{{\rm i}\theta},re^{-{\rm i}\theta}))\sin(\theta)\Big)}{-r+\epsilon\Big(\operatorname{Im}(L^\pm(re^{{\rm i}\theta},re^{-{\rm i}\theta}))\cos(\theta)-\operatorname{Re}(L^\pm(re^{{\rm i}\theta},re^{-{\rm i}\theta}))\sin(\theta)\Big)}=F^\pm(r,\theta,\epsilon).
\end{equation}
Hence, expanding $F^\pm(r,\theta,\epsilon)$ around $\epsilon=0$, \eqref{polar_eq} is written as
$$\frac{dr}{d\theta}=\epsilon F_1^\pm(\theta,r)+\mathcal{O}(\epsilon^2),$$
where $F_1^\pm(\theta,r)=-\operatorname{Re}(L^\pm(re^{{\rm i}\theta},re^{-{\rm i}\theta}))\cos(\theta)-\operatorname{Im}(L^\pm(re^{{\rm i}\theta},re^{-{\rm i}\theta}))\sin(\theta)$.
Computing the first averaged function
\begin{align*}
M_1^\pm(r)&=\displaystyle\int_0^{\pm\pi} F_1^\pm(\theta,r)d\theta\\
&=-\displaystyle\int_0^{\pm\pi} \operatorname{Im}({\rm i}\overline{L^\pm(re^{{\rm i}\theta},re^{-{\rm i}\theta})})\cos(\theta)-\operatorname{Im}(\overline{L^\pm(re^{{\rm i}\theta},re^{-{\rm i}\theta})})\sin(\theta) d\theta\\
&= -\displaystyle\int_0^{\pm\pi} \operatorname{Im}\Big(\overline{L^\pm(re^{{\rm i}\theta},re^{-{\rm i}\theta})}({\rm i}\cos(\theta)-\sin(\theta))\Big)d\theta\\
&= -\operatorname{Im}\left(\displaystyle\int_0^{\pm\pi}\overline{L^\pm(re^{{\rm i}\theta},re^{-{\rm i}\theta})}\,{\rm i}e^{{\rm i}\theta}d\theta\right)\\
&= -\operatorname{Im}\left(\displaystyle\int_{0}^{\pm\pi}\overline{\phi'(\phi^{-1}(re^{{\rm i}\theta}))R^\pm_m\left(\phi^{-1}(re^{{\rm i}\theta}),\overline{\phi^{-1}(re^{{\rm i}\theta}})\right)}\,{\rm i}e^{{\rm i}\theta}d\theta\right).
\end{align*}
From Proposition \ref{prop3} the result follows.

\subsection{The bifurcations functions $M_1$ and $N_1$}

The following proposition gives us the expressions of the Melnikov functions $M_1$ and $N_1$ as well as the maximum number of zeros that each on of these functions separately can have when  $f(z)={\rm i}(z^2-1)/2.$ The starting  point is to obtain first an explicit expression for $M_1$ by doing a detailed study around $z=-1$. Then, the analysis around $z=1$ will be reduced from the previous one. All these results are detailed in the Appendix~\ref{app}.

\begin{proposition}\label{prop_aux_complex}
	For system~\eqref{eq_pert} with $f(z)={\rm i}(z^2-1)/2$  and $m\le3,$ it holds that
	\begin{align*}
		M_1(r)&=\frac{1}{r}\Big(ar+br^2+cr^3+dr^4+\alpha(r^2-1)^2\operatorname{arctanh}(r)+\beta(r^4-1)\operatorname{arctanh}(r)\\
		&\qquad\quad+\gamma r^2\operatorname{arctanh}(r)\Big),\\
		N_1(r)&=\frac{1}{r}\Big(cr+(b+2d-\kappa+\rho)r^2+ar^3+(-d+\kappa)r^4+\alpha(r^2-1)^2\operatorname{arctanh}(r)\\
		&\qquad\quad-\beta(r^4-1)\operatorname{arctanh}(r)+\gamma r^2\operatorname{arctanh}(r)\Big),
	\end{align*} 
	where, for $m=3,$ the variables $a,b,c,d,\alpha,\beta,\gamma, \kappa, \rho$ can take any real value and depend linearly of the coefficients $a^\pm_{k,l}.$  When $m<3$ only appear the following restrictions: $\gamma=\rho=0$ when $m=2;$ $\gamma=\beta=\kappa=\rho=0$ when $m=1;$ and $\gamma=\beta=\alpha=\kappa=\rho=d=0$ and $c=-a$ when $m=0.$ More specifically, the values of these constants  are given in Remarks~\ref{re:M1} and~\ref{re:N1} in the Appendix.
	
	Moreover,  the maximum number of zeros of each $M_1$ and $N_1$ in $(0,1)$ is 1 when $m=0$ and it is $m+3$ when $1\le m\le 3.$ 
\end{proposition}
\begin{proof}  The expression of $M_1$ is given in Proposition~\ref{M1} and Remark~\ref{re:M1} of the Appendix~\ref{app}. We remark that we use the linearizing change $w=(1+z)/(1-z).$
	To get the expression of $N_1$ in terms of the coefficients of $M_1$, it is enough to use Proposition~\ref{N1} and Remark~\ref{re:N1} of that Appendix, by changing  $a_{k,l}^\pm$ in $M_1$  by $(-1)^{l}a_{k,l}^\mp,$ for all $0\leq k\leq l$ and $0\leq l\leq 3.$ Then  we arrive to the expression of $N_1$ of the statement and all the restrictions given there. We skip the details.
	
	Let us study the maximum number of zeros  of each of the functions $M_1$ and $N_1$ in $(0,1),$ separately, in terms of $m.$
	
	The case $m=0$ is simpler and we study it in a different way. In this situation 
	\[
	M_1(r)=a+br-ar^2,\qquad N_1=-a+br+ar^2,
	\]
	for arbitrary values $a=\operatorname{Im}(a_{ 0,0}^+)-\operatorname{Im}(a_{ 0,0}^-)$ and   $b=-\pi(\operatorname{Re}(a_{ 0,0}^-) + \operatorname{Re}(a_{ 0,0}^+))$. It is easy to see each of the functions has at most one zero in $(0,1).$
	
	Let us continue by studying the case $m=3.$ We want to prove that the ordered set of functions $$\mathcal{F}=[r,r^2,r^3,r^4,(r^2-1)^2\operatorname{arctanh}(r),(r^4-1)\operatorname{arctanh}(r),r^2\operatorname{arctanh}(r)]$$
	is an ECT-system in $(0,1).$  Notice that $M_1,N_1\in \operatorname{Span}($$\mathcal{F}$) . We will use Theorem~\ref{teo_wrons} together with  Lemma~\ref{wrons}. So, we need to compute several Wronkskians and prove that they do not vanish on $(0,1).$ We get,
	$W_0(r)=r,$ $W_1(r)=r^2,$ $W_2(r)=2r^3,$ $W_3(r)=12r^4,$ 
	\begin{align*}
		W_4(r)&=\frac{96(r(5 r^2-3)+3(r^2-1)^2\operatorname{arctanh}(r)  )}{(r^2-1)^2} ,\\	
		W_5(r)&=\frac{3072(-15r+22r^3-3r^5+3(r^2-1)^2 (5+r^2)\operatorname{arctanh}(r))}{(r^2-1)^6},\\
		W_6(r)&=\frac{294912r(-r(105-145r^2+15r^4+9r^6) +3(r^2-1)^2(35+10r^2+3r^4)\operatorname{arctanh}(r))}{(r^2-1)^{12}}.
	\end{align*}	
	Clearly, the Wronskians $W_j(r)\neq 0$ at $(0,1),$ for all $j\in\{0,1,2,3\}.$ Let us prove that $W_j(r)\neq 0,$ for $j=4,5,6$ on $(0,1).$ 
	
	By computing the first derivative of $W_4$ we obtain
	$$W_4'(r)=\frac{768r^4}{(1-r^2)^3}>0,$$ for all $r\in(0,1)$. Since $W_4(0)=0$ and $W_4$ is increasing at $(0,1)$, we  conclude that $W_4(r)>0$ for all $r\in(0,1)$.  
	
	Similarly, 
	\begin{align*}
		W_5'(r)&=\dfrac{18432(U(r)+V(r)\operatorname{arctanh}(r))}{(1-r^2)^7}>\dfrac{18432(U(r)+V(r)(r+\frac{r^3}{3}+\frac{r^5}{5}))}{(1-r^2)^7}\\[0.3cm]
		&=\dfrac{18432r^6(16+r^2+20r^4+3r^6)}{5(1-r^2)^7}>0,
	\end{align*}
	where $U(r)=-21r^2+32r^4-3r^6$ and $V(r)=3r(r^2-1)^2(7+r^2)$ and we have used that $\operatorname{arctanh}(r)> r+{r^3}/{3}+{r^5}/{5},$ for all $r\in(0,1)$, which can be proven for instance by using Taylor's formula. Again $W_5(0)=0$ and $W_5(r)$ is increasing, and so we can also conclude that $W_5(r)>0$ for all $r\in(0,1)$.  
	
	Finally,
	\begin{align*}
		W_6'(r)&=\dfrac{294912 \Big(U(r)+V(r)\operatorname{arctanh}(r)\Big)}{(1-r^2)^{13}}>\dfrac{294912 \Big(U(r)+V(r)(r+\frac{r^3}{3}+\frac{r^5}{5}+\frac{r^7}{7})\Big)}{(1-r^2)^{13}}\\[0.3cm]&=\dfrac{294912r^9 (255 + 1824 r^4 + 474 r^6 + 135 r^8)}{7(1-r^2)^{13}}>0,
	\end{align*}
	where $U(r)=-r(105+1910r^2-2864r^4+330r^6 +135r^8)$ and $V(r)=15(r^2-1)^2(7+139r^2+37r^4+9r^6),$ and this time we have used that $\operatorname{arctanh}(r)> r+{r^3}/{3}+{r^5}/{5}+{r^7}/{7},$ for all $r\in(0,1).$ As in the previous case  we get that $W_6(r)>0$ for all $r\in(0,1)$.

	According to Lemma \ref{wrons}, $\mathcal{F}$ is an ECT-system. Then, by Theorem \ref{teo_wrons}, $6=m+3$ is the maximum
	number of zeros for any element of $\operatorname{Span}(\mathcal{F})$ in $(0,1)$ and there are choices for $a,b,c,d,\alpha,\beta$ and $\gamma$ such that an element of $\operatorname{Span}(\mathcal{F})$  has exactly $j$ simple zeros in $(0,1)$ for any $j=0,1,\ldots,6.$

	From the expressions of the involved constants, the results when $m=1,$ (resp. $m=2$) correspond to $\kappa=\rho=\beta=\gamma=0$  (resp. $\rho=\gamma=0$) follow similarly. Moreover it is clear that the number of elements of $\mathcal{F}$ is $m+4,$ form an ECT-system on $(0,1)$ with at most $m+3$ zeros, taking into account their multiplicities, and the result follows.
\end{proof}

\section{Proof of Theorem \ref{tmb} and the polynomial case}

To study the number of zeros that the functions $M_1$ and $N_1$ can have in $(0,1)$ is a difficult question because of their transcendental nature. In Theorem~\ref{tmb} we only obtain partial results. In the first part of this section we will prove this theorem.

In the particular case  $\alpha=\beta=\gamma=0$ both functions become polynomial and the study is much more affordable. Although this case gives less limit cycles we include a detailed study because we believe that itself provides an interesting problem.

\begin{proof}[Proof of Theorem~\ref{tmb}]

Recall that by Proposition \ref{prop_aux_complex}, the bifurcation function of PWHS~\eqref{eq_pert}  associated to $z=-1$ and $z=1$ are given respectively by $M_1(r)=\frac{1}{r}f(r)$ and $N_1(r)=\frac{1}{r}g(r),$ where
\begin{align*}
f(r)&=ar+br^2+cr^3+dr^4+\alpha(r^2-1)^2\operatorname{arctanh}(r)\\&\quad+\beta(r^4-1)\operatorname{arctanh}(r)+\gamma r^2\operatorname{arctanh}(r),\\
	g(r)&=cr+(b+2d-\kappa+\rho)r^2+ar^3+(-d+\kappa)r^4+\alpha(r^2-1)^2\operatorname{arctanh}(r)\\
	&\quad-\beta(r^4-1)\operatorname{arctanh}(r)+\gamma r^2\operatorname{arctanh}(r).
\end{align*} 
with $a,b,c,d,\alpha,\beta,\gamma,\kappa$ and $\rho$ real coefficients, which depend of coefficients of the system and of~$m.$ Moreover, the maximum number of zeros that these functions separately can have is $1$ when $m=0$ and $m+3$ otherwise. Let us study the number of simple simultaneous  zeros that $f$ and $g$ possess in $(0,1)$ in several situations. Clearly, these zeros coincide  with the  of zeros in $(0,1)$ of $M_1$ and $N_1,$ respectively. Recall that we denote these number of zeros as $m_1,n_1,$ respectively, and they give the realizability of the configuration $[[m_1,n_1]]$ for  PWHS~\eqref{eq_pert}.

We will fix $m\in\{0,1,2,3\}$ and prove the result case by case. 

When $m=0$  the proof is very simple. We know that 	
$M_1(r)=a+br-ar^2$ and  $N_1(r)=-a+br+ar^2,$ 
for arbitrary values $a,b\in\mathbb{R}$. Then, it is easy to see that either $m_1=n_1=0$ or one of the values is 0 and the other one is 1.

For the cases $1\le m\le 3,$ we will not give all details, but a procedure that allows to control the  number of zeros of $f$ and $g$ by forcing the existence of several zeros of them in $(-1,1).$

(a) When $m=1,$ then $\kappa=\rho=\beta=\gamma=0.$ The case $\alpha=0$ is much simpler and will be studied in next Proposition~\ref{pro:pol}, so we consider $\alpha\ne0$ and we can
assume that $\alpha=1.$ 

Then, the main idea of our approach is to consider four different values $r_1,r_2,r_3$ and $r_4$ in $(-1,1)$ and then impose that four equations among the eight ones: $f(r_j)=0, g(r_j)=0, j=1, 2, 3, 4$ are fulfilled.  Then these four equations fix  the values of $a,b,c$ and $d$ and it is easy to obtain them even explicitly, because the eight equations are linear with unknowns $a,b,c,d.$

In this way,  any of the configurations $[[m_1,n_1]]$ with $m_1+n_1\le4,$  can be obtained. Notice that the negative values of $r_j$ give zeros of $f$ or $g$ that do not contribute to any of the values $m_1$ or $n_1,$ because only simple positive zeros give rise to limit cycles of PWHS~\eqref{eq_pert}.

Notice also that when all the procedure is applied, to be sure that a configuration happens we need to prove that the forced zeros are simple. This is not always an easy task but it can be done with a case by case study. For instance, if for $0<m\le 3,$  $m_1=m+3,$ then they are always simple zeros of $f$ because this function is an element of an ECT system. If $m_1<m+3$ and not all zeros were simple, then it is easy to perturb the function to have at least $m_1$ simple zeros. Afterwards, one has to take care of the zeros of $g.$ Each situation needs special tricks and sometimes some careful computations. Finally, it has to be studied if the given zeros are  the only ones in $(0,1)$ or some extra zero does appear. Although this could be done, again by a case by case study, we do not give details on this matter. The main reason is that without studying this last question we already know that at least $m_1$ limit cycles surrounding $z=-1$ and  $n_1$ limit cycles surrounding $z=1$ exist although, eventually, more limit cycles could also appear.

As an illustration we present  a detailed study for the case $(m_1,n_1)=(4,0).$ Fix $r_j=j/5,$ for $j=1,2,3,4$ and force that all these values are zeros of $f.$ This  completely fixes the parameters $a,b,c$ and $d.$ Moreover, since we know that $[r,r^2,r^3,r^4,(r^2-1)^2\operatorname{arctanh}(r)]$ is an ECT in $(0,1)$ we can ensure that these zeros are simple for $f.$ Then we have to prove that $g$ does not have zeros in $(0,1).$ With this aim, it can be seen that
\[
\left(\frac{g(r)}{(r^2-1)^2} \right)'= \frac{P_4(r)}{(r^2-1)^3},
\]
where $P_4$ is a fixed polynomial of degree 4. It can be seen, by computing its Sturm's sequence, that it is positive in $[0,1].$ Since $g(0)=0,$ this shows that $g(r)>0$ in $(0,1)$ and $n_1=0,$ as we wanted to prove.

(b) When $m=2,$ $\rho=\gamma=0.$ Recall that in this case  the maximum number of zeros that the functions $f$ and $g$  separately can have in $(0,1)$ is five. We will look for new configurations not appearing when $m<2.$  Similarly that in the case $m=1,$ we can fix $\beta=1,$ take six different values $r_j, j=1,\ldots,6$ in $(-1,1)$ and impose that six  equations among the twelve ones: $f(r_j)=0, g(r_j)=0, j=1,\ldots,6$ are fulfilled. We remark that at most five of these equations can involve $f$ or $g.$  Then from these six equations we obtain explicitly  the values of $a,b,c,d, \kappa$ and $\alpha.$  They provide all configurations with  $0\le m_1+n_1\le 6$ with $m_1,n_1\le5.$

(c) Case $m=2.$  In this occasion we fix $\gamma=1$ and eight values between $-1$ and $1$ fix the parameters $a,b,c,d,\kappa,\rho,\alpha$ and $\beta.$ By using this approach we obtain all configurations with  $0\le m_1+n_1\le 8$ with $m_1,n_1\le6.$
\end{proof}

Next proposition fully characterizes the number of simultaneous zeros of $M_1$ and $N_1$ when both functions are polynomial.

\begin{proposition}\label{pro:pol} For
each $m\le3,$ set
\begin{align*}
	M_1(r)=a+br+cr^2+dr^3,\quad\mbox{and}\quad
	N_1(r)=c+(b+2d-\kappa+\rho)r+ar^2+(-d+\kappa)r^3,
\end{align*} 
the functions given in  Proposition~\ref{prop_aux_complex} when $\alpha=\beta=\gamma=0.$ Let $m_1\le 3$ and $n_1\le 3$ be, respectively, their number of zeros in $(0,1)$ taking into account their multiplicities. Then the following holds:
\begin{enumerate}
	\item[(i)] When $m=0,$ ($\kappa=\rho=d=0,$ $c=-a$)  then $m_1+n_1\le 1.$
	\item[(ii)] When $m=1,$ ($\kappa=\rho=0$)  then  $m_1+n_1\le 4.$ 
     \item[(iii)]  When $m=2,$ ($\rho=0$)  then  $m_1+n_1\le 4.$
     \item[(iv)] When $m=3,$ then $m_1+n_1\le 5.$ 
\end{enumerate}
Moreover, all values of $m_1$ and $n_1$ satisfying  the above restrictions are attained, except $(m_1,n_1)\in\{(3,0),(0,3)\}$ when $m\in\{1,2\}.$
\end{proposition}

\begin{proof}  (i) It is already proved in Theorem~\ref{tmb}.

When $m\ne0,$ the most interesting and difficult case happens if either $m_1=3$ or $n_1=3.$ We will concentrate in the case $m_1=3,$ because the other situation can be reduced to this one.

In particular,  $d$ must be non zero and without loss of generality we can assume that $d=1.$ Since $m_1=3,$ $M_1$ has all its roots $r_1,r_2$ and $r_3$  in $(0,1),$ we obtain that
   \begin{align*}
	M_1(r) &=r^3+cr^2+br+a=(r-r_1)(r-r_2)(r-r_3)\\&=r^3-(r_1+r_2+r_3)r^2+(r_1r_2+r_1r_3+r_2r_3)r-r_1r_2r_3.
\end{align*}
Then,	
	\begin{align*}
	N_1(r)= (\kappa-1) r^3 -r_1r_2r_3 r^2+\big(r_1r_2+r_1r_3+r_2r_3+2-\kappa +\rho \big)r   -(r_1+r_2+r_3).
\end{align*}
Notice that
\begin{align*}
N_1(0)=-(r_1+r_2+r_3)<0,\quad N_1(1)=(1-r_1)(1-r_2)(1-r_3)+\rho.
\end{align*}

Let us prove item~(ii). When  $ r_1,r_2,r_3\in(0,1)$ and $\kappa=\rho=0,$ then $N_1(0)<0,$ $N_1(1)>0$ and 
	$N_1''(r)=-6r-2r_1r_2r_3<0$ for $r\ge0.$ Hence when $m_1=3,$ by Bolzano's Theorem $n_1\ge1$ and by Rolle's Theorem $n_1\le2$ because $N_1''\big|_{[0,1]}<0.$  Moreover, the possibility $n_1=2$ is incompatible with $N_1(0)N_1(1)<0$ and it holds that $(m_1,n_1)=(3,1).$ 
	
	All the other cases satisfying $m_1<3$ or $n_1<3$ and $0\le m_1+n_1\le 4$ can be easily obtained by simple inspection. For instance, by taking $r_1=1/6$, $r_2=1/4$ and $r_3\in\mathbb{R}\setminus(0,1)$ as the roots of $M_1$, we get that  for $r_3= -2,-1/5, 21/20,$ it holds that  $m_1=2$ and the values of $n_1$ are $0,1$ and $2,$ respectively. We omit the other examples.

(iii) By using item~(ii) it is clear that all cases with $0\le m_1+n_1\le4$ and $m_1<3$ and $n_1<3$ do happen. It is also clear that there are examples where  $(m_1,n_1)$ is $(3,1)$ or $(1,3).$  Let us prove that, as in the above case, when $m_1=3$  then $n_1=1.$ In this case the proof is more involved.

As in item~(ii) we assume that  $m_1=3$ and  $s_1,s_2,s_3\in(0,1).$ In this case $\rho=0$ and also happens that $N_1(0)N_1(1)<0,$ because it is independent of $\kappa.$ In particular we know that $n_1\ge1.$ The difference with the above case is that $N_1''(r)=6(\kappa-1)r-2r_1r_2r_3$ and when $\kappa>1$ this function can change sign in $(0,1).$ In any case, when $\kappa\le 1$ we know that $n_1=1$ and the result follows.

To prove that $n_1=1$ when $\kappa>1,$  we will apply Lemma~\ref{roots_1pf}. We fix the values $ r_1,r_2,r_3\in(0,1),$  consider  $\kappa$ as a parameter and introduce the notation 
\[Q_\kappa(r):=N_1(r)= (\kappa-1) r^3 -r_1r_2r_3 r^2+\big(r_1r_2+r_1r_3+r_2r_3+2-\kappa\big)r   -(r_1+r_2+r_3)
.\] Notice that 
\[
Q_\kappa(0)=N_1(0)<0\quad\mbox{and}\quad Q_\kappa(1)=N_1(1)>0.
\]
To apply the lemma we first need to study the zeros of $\Delta_r(Q_\kappa(r)).$ 
Some computations give that  $$\Delta_r(Q_\kappa(r))=4\kappa^4+\eta_3 \kappa^3+ \eta_2 \kappa^2+\eta_1\kappa+\eta_0,$$ where the coefficients $\eta_j=\eta_j(r_1,r_2,r_3)$, $j=0,1,2,3,$ are symmetric polynomials that we skip for the sake of shortness. 
It is well-know that given any real quartic polynomial $P(\kappa)$ such that $\Delta_\kappa (P(\kappa))<0$ it has  two real roots and two complex ones, see~\cite{Rees}.  Some tedious computations give that
$$\Delta_\kappa \big(\Delta_r(Q_\kappa(r))\big)=-256(r_1^2-1)(r_2^2-1)( 
r_3^2-1)(r_1 + r_2 + r_3)(E(r_1,r_2,r_3))^3,$$
with
\begin{align*}
	E(r_1,r_2,r_3)&=-27 (r_2 + r_3) (1 + r_2 r_3)^2 + 
	27 r_1 (-1 + r_2^3 r_3 - 2 r_3^2 + r_2 r_3 (-6 + r_3^2)  \\
	&\quad+ 
	r_2^2 (-2 + r_3^2)) + r_1^3 (27 r_2 r_3 - 27 r_3^2 + r_2^3 r_3^3 + 9 r_2^2 (-3 + 2 r_3^2))\\
	&\quad  + 9 r_1^2 (3 r_2 (-2 + r_3^2) - 3 r_3 (2 + r_3^2) + r_2^3 (-3 + 2 r_3^2) + r_2^2 r_3 (3 + 2 r_3^2)),\\
\end{align*}
which can be seen that is negative for all $(r_1,r_2,r_3)\in(0,1)^3.$ This is so, because by studying the system
\[
\frac{\partial}{\partial r_1}E(r_1,r_2,r_3)=0,\quad \frac{\partial}{\partial r_2}E(r_1,r_2,r_3)=0,\quad \frac{\partial}{\partial r_3}E(r_1,r_2,r_3)=0,
\]
we get that does not have solutions in $(0,1)^3.$ Hence,
 the maximum of the function $E$ on the box $[0,1]^3$ is 0 and it is reached in the boundary at the point $(0,0,0).$ 
Furthermore, it is easy to see that the two real zeros of $\Delta_r(Q_\kappa(r)),$  $\kappa_1$ and $\kappa_2$ satisfy
\[
1<\kappa_1<R< \kappa_2,\quad \mbox{where} \quad R= r_1r_2+r_1r_3+r_2r_3+2,
\]
because  $\Delta_r(Q_\kappa(r))|_{\kappa=1}>0$ and $\Delta_r(Q_\kappa(r))|_{\kappa=R}<0.$

Hence, if we define the three intervals $K_1=(1,\kappa_1),$ $K_2=(\kappa_1,\kappa_2)$ and $K_3=(\kappa_2,\infty),$
 by Lemma~\ref{roots_1pf} the value $n_1$ (that is the number of roots of $Q_\kappa(r)$ in $(0,1)$) when $\kappa\in K_j,$ $j=1,2,3$ does not depend on $\kappa,$ but on $j$ and maybe on the values of $r_1,r_2$ and $r_3.$

Similarly that in the quartic case, it is also well-know that given any real cubic polynomial $P(r)$ it holds that:
\begin{itemize}
	\item If $\Delta_r(P(r))>0$ it has  three simple real roots; and
		\item If $\Delta_r(P(r))<0$ it has one simple real root and two simple complex roots.
\end{itemize}   
Hence, if we take $\kappa\in K_1,$ then  $\Delta_r(Q_\kappa(r))>0$ and  $Q_\kappa$ has three simple real roots. Let us prove that two of them are greater than  $1.$ To ensure that $\kappa\in K_1$ we take $\kappa=1+\varepsilon,$ for $\varepsilon>0,$ small enough. For this value of $\kappa,$ let us prove that $Q_\kappa$ has a positive root that tends to infinite when $\epsilon$ tends to 0. To prove this fact consider the new variable $s=1/r,$ Then, when $\kappa=1+\varepsilon,$
\[
P_\varepsilon(s):=s^3 Q_\kappa(1/s)=  -(r_1+r_2+r_3)s^3 +\big(r_1r_2+r_1r_3+r_2r_3+1-\varepsilon\big)s^2  -r_1r_2r_3 s+\varepsilon. 
\]
By the implicit function Theorem $P_\varepsilon$ has a zero $s(\varepsilon)= \varepsilon/(r_1r_2r_3)+O(\varepsilon^2)$ that tends to zero when $\varepsilon$ tends to zero. This zero gives a positive zero $r(\varepsilon)$ of  $Q_{1+\varepsilon}$ that tends to infinity when $\varepsilon$ goes to zero. Moreover its asymptotic expansion at $\varepsilon=0$ is $r(\varepsilon)\sim r_1r_2r_3/\varepsilon.$ Hence, from the existence of two positive roots of $Q_{1+\varepsilon},$ one in $(0,1)$ and a second one near infinity we deduce the existence of a third one, which moreover it is in $(1,r(\varepsilon)),$ as we wanted to prove.
 
 If we take $\kappa\in K_2,$ then  $\Delta_r(Q_\kappa(r))<0$ and  $Q_\kappa$ has a single real root. Since $Q_\kappa$ has a root in $(0,1)$ then $n_1=1.$
 
  If we take $\kappa\in K_3,$ then again $\Delta_r(Q_\kappa(r))>0$ and  $Q_\kappa$ has three simple real roots. To know the localization of the roots it suffices to consider a value of $\kappa$ big enough. Then the signs of the ordered coefficients of $Q_\kappa$ are $[+,-,-,-]$ and  by Descarte's rule of signs (see Theorem \ref{descartes}) $Q_r$ has exactly one positive root. Hence, as in the previous case  $n_1=1.$

In short, when $\kappa\not\in\{\kappa_1,\kappa_2\}$ it holds that $(m_1,n_1)=(3,1).$ Otherwise, some multiple root of~$N_1$ appears but never in $(0,1).$ 

(iv) We only need to take care of  cases with at least  five zeros.
Let us assume that $m_1=3$ and $n_1\ge2$ and prove that indeed $n_1=2.$ As in the previous case, the values $0<r_1,r_2,r_3<1$ fix $M_1.$ By imposing that $0<s_1,s_2<1$ and $N_1(s_1)=N_1(s_2)=0$ we obtain that
\begin{align*}
	\kappa &= \frac{r_1r_2r_3s_1s_2+ s_1s_2(s_1+s_2)-(r_1+r_2+r_3)}{  s_1s_2(s_1+s_2)},\\
	\rho&=\frac{U(r_1,r_2,r_3,s_1,s_2)}{  s_1s_2(s_1+s_2)},\\
U(r_1,r_2,r_3,s_1,s_2)&= (r_1+r_2+r_3)(s_1^2+s_2^2+s_1s_2-1)\\&\quad-(r_1r_2+r_1r_3+r_2r_3+1)s_1s_2(s_1+s_2) +r_1r_2r_3s_1s_2(s_1s_2+1).
\end{align*} 
Then
\[
N_1(r)=\frac{(r-s_1)(r-s_2)}{  s_1s_2(s_1+s_2)} \big((r_1r_2r_3s_1s_2-r_1-r_2-r_3)r-(s_1+s_2)(r_1+r_2+r_3)\big)
\] 
and the third root of $N_1$ is
\[
r_3= \frac{(s_1+s_2)(r_1+r_2+r_3)}{r_1r_2r_3s_1s_2-r_1-r_2-r_3}<0,
\] 
because $ r_1r_2r_3s_1s_2<r_1$ and so $r_1r_2r_3s_1s_2-r_1-r_2-r_3<-(r_2+r_3)<0.$ Hence, since $r_3\notin(0,1),$ $n_1=2.$
\end{proof}

\section{Proof of Theorem \ref{tmc}}
The next result provides us with the expressions of the Melnikov functions $M_1$ and $N_1$ as well as the maximum number of zeros that these functions separately can have in the PWHS case when $f(z)={\rm i}(z^2-1)/2,$ and $m\ge 3.$
\begin{proposition}\label{prop_aux_hol}
    Let $R^\pm_m$ be a holomorphic polynomial of degree $m$ in \eqref{eq_pert}  when $f(z)={\rm i}(z^2-1)/2,$ and $m\ge 3.$ Then, the Melnikov functions $M_1$ and $N_1$  on -1 and 1 associated to it are:
\begin{align*}
	M_1(r)&=\displaystyle\frac{1}{(r^2-1)^{m-3}}\left[\sum_{n=0}^{2(m-2)}a_nr^n+\alpha r(r^2-1)^{m-3}\operatorname{arctanh}(r)\right],\\  N_1(r)&=\displaystyle\frac{1}{(r^2-1)^{m-3}}\left[\sum_{n=0}^{2(m-2)}b_nr^n+\alpha r(r^2-1)^{m-3}\operatorname{arctanh}(r)\right],
\end{align*}
     where $a_n,$ $b_n$ and $\alpha$ depend of the coefficients $a^\pm_{0,l}$ for $l=\{0,\dots,m\}.$  
     Moreover there are several linear relations among the values~$a_n$ and the values~$b_n$ as can be seen in the proof.
\end{proposition}
\begin{proof} 

Since the functions $R_m^\pm$ are holomorphic, then  $\overline{a}_{k,l}=0$ for all $1\leq k\leq l$ and $0\leq l\leq m$. Thus, from the formula~\eqref{rec_for} we get that
\begin{equation}\label{rec_for_hol}
	M_1^\pm(r)=\sum_{l=0}^m\left[\sum_{k=0}^l I_{k,l}^\pm(r)\right]=\sum_{l=0}^m I_{0,l}^\pm(r).
\end{equation}
The expressions of $I_{0,l}^\pm$ when $l\le 3$ are already detailed in the proof of Proposition~\ref{M1} in the general situation and can be easily particularized to the holomorphic case. Straightforward calculations allows us to get that for $l\ge3,$
 \[I^\pm_{0,l}(r)=-\operatorname{Im}(a^\pm_{0,l})\left(\frac{ P_{2(l-2)}(r)}{(r^2-1)^{l-3}}+\eta(l)r\operatorname{arctanh}(r)\right)\pm(-1)^{l}\operatorname{Re}(a^\pm_{0,l})(l-1)\pi r,\] 
where $P_{2(l-2)}$ is a polynomial function of degree $2(l-1),$ with rational coefficients, and $\eta(l)= 2(1-l)\big(1-(-1)^l\big).$ Then,
  \begin{align*}
M_1^\pm(r)&=\displaystyle\sum_{l=0}^2 I_{0,l}^\pm(r)+\displaystyle\sum_{l=3}^m I_{0,l}^\pm(r)
= -\displaystyle\sum_{l=0}^2\operatorname{Im}(a_{0l}^\pm)((-1)^{l+1}+r^2)\pm\displaystyle\sum_{l=0}^m(-1)^{l}\operatorname{Re}(a^\pm_{0,l})(l-1)\pi r\\
&\quad-\displaystyle\sum_{l=3}^m\operatorname{Im}(a^\pm_{0,l})\left(\frac{ P_{2(l-2)}(r)}{(r^2-1)^{l-3}}+\eta(l)r\operatorname{arctanh}(r)\right) \\
&= \dfrac{1}{(r^2-1)^{m-3}}\left(-\displaystyle\sum_{l=0}^2\operatorname{Im}(a_{0l}^\pm)((-1)^{l+1}+r^2)(r^2-1)^{m-3}\right.\\
&\quad\qquad\qquad\qquad \pm\displaystyle\sum_{l=0}^m(-1)^{l}\operatorname{Re}(a^\pm_{0,l})(l-1)\pi r(r^2-1)^{m-3}\\
&\quad\qquad\qquad\qquad \left.-\displaystyle\sum_{l=3}^m\operatorname{Im}(a^\pm_{0,l})\left(P_{2(l-2)}(r)(r^2-1)^{m-l}+\eta(l)r(r^2-1)^{m-3}\operatorname{arctanh}(r)\right)\right).
  \end{align*}
 Thus, 
  \begin{align*}
 M_1(r)&= M_1^+(r)- M_1^-(r)=\displaystyle\frac{1}{(r^2-1)^{m-3}}\left[\sum_{n=0}^{2(m-2)}a_nr^n+\alpha r(r^2-1)^{m-3}\operatorname{arctanh}(r)\right],
  \end{align*}
 where $\alpha=-\sum_{l=3}^{m}(\operatorname{Im}(a^+_{0,l})-\operatorname{Im}(a^-_{0,l}))\eta(l)$ and $a_n$ depends of the coefficients $a^\pm_{0,l}$ for $l=\{0,\dots,m\}$.

  To obtain the expression of $N_1$ of the statement  in terms of the coefficients of $M_1$, it is enough to use Proposition~\ref{N1} of Appendix~\ref{app}. 
 \end{proof}

\begin{proof}[Proof Theorem \ref{tmc}]

(a) In the case $m=0$ the  PWCS is indeed holomorphic and the proof is the same as that given in Theorem~\ref{tmb}(a).
   
From Proposition \ref{prop_aux_hol}, when $m\in\{1,2\},$ the bifurcation function of PWHS   \eqref{eq_pert} associated to $z=-1$ and $z=1$ are given respectively by
   $M_1(r)=a+br+cr^2$
   and 
   $N_1(r)=c+br+ar^2= r^2\big(M_1(1/r)\big)$
where $a,b,c$ are arbitrary real numbers. Hence if $r=r^*$ is root of $M_1$ then $r=1/r^*$ is a root of $N_1$ and vice versa because $M_1$ and $N_1$ are reciprocal polynomials. Hence the only possible configurations are $[[i,j]]$ with $i+j\le 2$ and $i,j\le2$ and all them are realizable.

(b) When $m=3,$ according Proposition \ref{prop_aux_hol}, the bifurcation functions of the PWHS \eqref{eq_pert} associated to $z=-1$ and $z=1$ are given respectively by
 \begin{align*}M_1(r)&=a+b r+cr^2+\alpha r\operatorname{arctanh}(r),\\
   N_1(r)&=c+(b-\kappa)r+ar^2+\alpha r\operatorname{arctanh}(r),\end{align*}
where $a,b,c,\alpha$ and $\kappa$ are arbitrary real coefficients, which depend of the real and imaginary parts of  $a^\pm_{0,l}$, for all $l=0,1,2,3$. Even more, the maximum number of zeros that these functions separately can have is 3, because it can be seen that the functions $[1,r,r^2,r\operatorname{arctanh}(r)]$ form an ECT-system in $(0,1).$ Indeed this property also follows from the computations done in next item (c) by taking $m=3.$  By using the same tools that in the proof of item (b) of Theorem~\ref{tmb} we obtain that all configurations $[[i,j]]$ with $i,j\le3,$ and $i+j\le4,$ are realizable.

(c)  By Proposition~\ref{prop_aux_hol}, the first order averaged functions $M_1$ and $N_1,$ multiplied by $(r^2-1)^{m-3},$ belong to the vectorial space $\mathcal{G}$ generated by the ordered set of functions 
$$\mathcal{G}=[1,r,r^2,\ldots, r^{2(m-2)}, r(r^2-1)^{m-3}\operatorname{arctanh}(r)].$$ Let us prove that they form an ECT-system on $(0,1).$ Their Wronskians, defined in Lemma~\ref{wrons}, are
$W_j(r)=\prod_{k=0}^{j}k!\neq 0$  for all $j\in\{0,1,2,\dots,2(m-2)\}$ and $$W_{2(m-2)+1}(r)=\frac{(-1)^m\xi_m r}{(r^2-1)^{m}}\neq 0,$$ at $(0,1),$ where $\xi_m$ is an increasing sequence of positive real numbers.  Then, according to this lemma,  $\mathcal{G}$ is an ECT-system formed $2m-4$ elements. Then, by Theorem~\ref{teo_wrons},  $2m-3$ is the maximum
number of roots in $(0,1)$ for any element of $\operatorname{Span}(\mathcal{G}),$ taking into account their multiplicities, as we wanted to prove. Notice that at this point, to prove that there are values of $a^\pm_{0,l}$ for $l=\{0,\dots,m\},$ for which the corresponding piecewise holomorphic system has $2m-3$ nested limit cycles surrounding $z=-1$ it would suffice to show that there is a choice of these parameters such that  $a_n$, and $\alpha$, $n\in\{0,\dots,2(m-2)+1\}$ can take arbitrary values. 
\end{proof}

\section{Appendix}\label{app}

This appendix is devoted to find the explicit expressions of the first order averaged functions  $M_1$ and $N_1$ for PWCS~\eqref{eq_pert} when $f(z)={\rm i}(z^2-1)/2$ when $m\le3.$
We will start by doing a detailed study of $M_1$ around $z=-1$. The analysis of $N_1$ around $z=1$ will be deduced from the previous one.

\begin{proposition}\label{M1} For system~\eqref{eq_pert}  when $f(z)={\rm i}(z^2-1)/2$ and $m\le3,$ it holds that
	\begin{align*}M_1(r)&=\frac{1}{r}\Big(ar+br^2+cr^3+dr^4+\alpha(r^2-1)^2\operatorname{arctanh}(r)\\&\qquad+\beta(-1+r^4)\operatorname{arctanh}(r)+\gamma r^2\operatorname{arctanh}(r)\Big),\end{align*} where the variables $a,b,c,d,\alpha,\beta$ and $\gamma$ can take any real value for $m=3$. When $m<3$ only appear the following restrictions: $\gamma=0,$ when $m=2;$ $\gamma=\beta=0$ when $m=1;$ and $\gamma=\beta=\alpha=0$ and $c=-a$ when $m=0.$ More specifically, the values of these constants  are given in Remark~\ref{re:M1}.
\end{proposition}

\begin{proof}  To employ Theorem \ref{tma} to \eqref{eq_pert} at $z=-1$ we must first linearize $\dot{z}={\rm i}(z^2-1)/2$. It is easy to verify that if
	\begin{equation}\label{phi_features}
		\phi(z)=\frac{1+z}{1-z},\quad \phi'(z)=\frac{2}{(z-1)^2}\quad \text{and}\quad\phi^{-1}(w)=\frac{w-1}{w+1},
	\end{equation}
	and by taking $w=\phi(z)$ the differential equation writes as $\dot w= iw.$
	See the behaviour of the conformal map $\phi$ in Figure~\ref{fig_conf_map}.
	\begin{figure}[h]
		\begin{overpic}[scale=0.5]{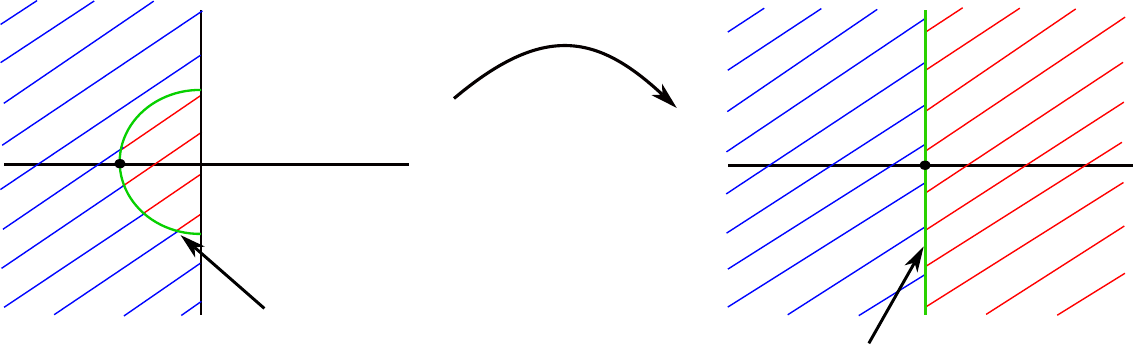}
			\put(102,15){$\phi(\Sigma)$}
			\put(38,15){$\Sigma$}
			\put(24,0){$\mathbb{S}^1$}
			\put(74,-3){$\phi(\mathbb{S}^1)$}
			\put(49,27){$\phi$}
		\end{overpic}
		\caption{\footnotesize{Conformal map $\phi(z)=\frac{1+z}{1-z}$.}}
		\label{fig_conf_map}
	\end{figure}
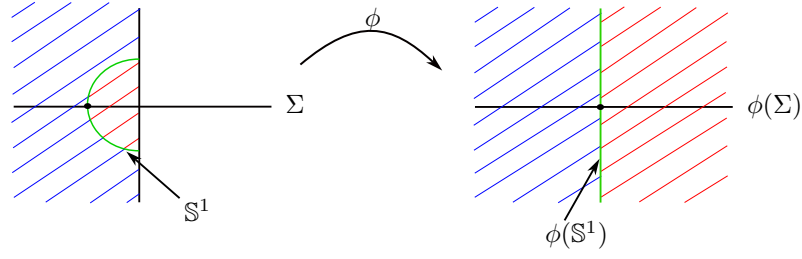

	In addition, from \eqref{phi_features}, we have that
	$$\phi'\left(\phi^{-1}(w)\right)=\frac{1}{2}(w+1)^2,\qquad\overline{\phi^{-1}(w)}=\frac{\overline{w}-1}{\overline{w}+1},$$	
	and	$\phi(\Sigma)=\Sigma$ and $\phi(\mathbb{S}^1)=\{(0,y):y\in\mathbb{R}\}$.
	
	From Theorem \ref{tma} with $R^\pm(z,\overline{z})=R_m^\pm(z,\overline{z})$ given in \eqref{tom} and $f(z)={\rm i}(z^2-1)/2$ we get that  
	\begin{equation}\label{rec_for}
		M_1(r)=	M_1^+(r)-	M_1^-(r),\quad\mbox{where}\quad	M_1^\pm(r)=\sum_{l=0}^m\Big(\sum_{k=0}^lI_{k,l}^\pm(r)\Big),
	\end{equation}
	and
	$$I^\pm_{k,l}(r)=-\operatorname{Im}\left(a_{k,l}^\pm\displaystyle\int_{0}^{\pm\pi}\overline{\phi'(\phi^{-1}(re^{{\rm i}\theta}))}\left(\phi^{-1}(re^{{\rm i}\theta})\right)^k\left(\overline{\phi^{-1}(re^{{\rm i}\theta})}\right)^{l-k}{\rm i}e^{{\rm i}\theta}d\theta\right).$$
	Thus, using \eqref{rec_for} we get
	\begin{align*}
		I^\pm_{k,l}(r)&=-\operatorname{Im}\left(a_{k,l}^\pm\displaystyle\int_{0}^{\pm\pi}\frac{1}{2}(re^{-{\rm i}\theta}+1)^2\left(\frac{re^{{\rm i}\theta}-1}{re^{{\rm i}\theta}+1}\right)^k\left(\frac{re^{-{\rm i}\theta}-1}{re^{-{\rm i}\theta}+1}\right)^{l-k}{\rm i}e^{-{\rm i}\theta}d\theta\right)\vspace{.3cm}\\
		&=-\dfrac{1}{2}\operatorname{Im}\left(a_{k,l}^\pm\displaystyle\int_{0}^{\pm\pi}\dfrac{(re^{-{\rm i}\theta}+1)^{k-l+2}(re^{{\rm i}\theta}-1)^k(re^{-{\rm i}\theta}-1)^{l-k}{\rm i}e^{{\rm i}\theta}}{(re^{{\rm i}\theta}+1)^k}d\theta\right).
	\end{align*}	
	To arrive to the final expression of $M_1$ we have to compute each of the functions $I^\pm_{k,l}.$ When $l=0,$
	\begin{align*}
		I^\pm_{0,0}(r)&= -\dfrac{1}{2}\operatorname{Im}\left(a_{0,0}^\pm\displaystyle\int_0^{\pm\pi}(re^{-{\rm i}\theta}+1)^2 {\rm i}e^{{\rm i}\theta}d\theta\right)\vspace{0.3cm}\\
		&= -\operatorname{Im}\left(a_{0,0}^\pm\right)(r^2-1)\mp\pi \operatorname{Re}\left(a_{0,0}^\pm\right)r.
	\end{align*}
	For $l=1$ we obtain,
	\begin{align*}
		I^\pm_{0,1}(r)&= -\dfrac{1}{2}\operatorname{Im}\left(a_{0,1}^\pm\displaystyle\int_0^{\pm\pi}(re^{-{\rm i}\theta}+1)(re^{-{\rm i}\theta}-1) {\rm i}e^{{\rm i}\theta}d\theta\right)= -\operatorname{Im}\left(a_{0,1}^\pm\right)(1+r^2),\vspace{0.3cm}\\
		I^\pm_{1,1}(r)&= -\dfrac{1}{2}\operatorname{Im}\left(a_{1,1}^\pm\displaystyle\int_0^{\pm\pi}\frac{(re^{-{\rm i}\theta}+1)^2(re^{{\rm i}\theta}-1) {\rm i}e^{{\rm i}\theta}}{re^{{\rm i}\theta}+1}d\theta\right)\vspace{0.3cm}\\
		&= -\dfrac{\operatorname{Im}\left(a_{1,1}^\pm\right)}{r}\Big(-r(1+r^2)+2(r^2-1)^2\operatorname{arctanh}(r)\Big)\mp\pi\operatorname{Re}\left(a_{1,1}^\pm\right)r(r^2-1).\\
	\end{align*}
	When $l=2$ we arrive to 
	\begin{align*}
		I^\pm_{0,2}(r)&= -\dfrac{1}{2}\operatorname{Im}\left(a_{0,2}^\pm\displaystyle\int_0^{\pm\pi}(re^{-{\rm i}\theta}-1)^2 {\rm i}e^{{\rm i}\theta}d\theta\right)\vspace{0.3cm}\\
		&= -\operatorname{Im}\left(a_{0,2}^\pm\right)(r^2-1)\pm\pi\operatorname{Re}\left(a_{0,2}^\pm\right) r,\vspace{0.3cm}\\
		I^\pm_{1,2}(r)&= -\dfrac{1}{2}\operatorname{Im}\left(a_{1,2}^\pm\displaystyle\int_0^{\pm\pi}\frac{(re^{-{\rm i}\theta}+1)(re^{{\rm i}\theta}-1)(re^{-{\rm i}\theta}-1) {\rm i}e^{{\rm i}\theta}}{re^{{\rm i}\theta}+1}d\theta\right)\vspace{0.3cm}\\
		&= -\dfrac{\operatorname{Im}\left(a_{1,2}^\pm\right)}{r}\Big(r-r^3+2(-1+r^4)\operatorname{arctanh}(r)\Big)\mp\pi\operatorname{Re}\left(a_{1,2}^\pm\right)r^3,\\
		I^\pm_{2,2}(r)&= -\dfrac{1}{2}\operatorname{Im}\left(a_{2,2}^\pm\displaystyle\int_0^{\pm\pi}\frac{(re^{-{\rm i}\theta}+1)^2(re^{{\rm i}\theta}-1)^2 {\rm i}e^{{\rm i}\theta}}{(re^{{\rm i}\theta}+1)^2}d\theta\right)\vspace{0.3cm}\\
		&= -\dfrac{\operatorname{Im}\left(a_{2,2}^\pm\right)}{r}\Big(5r(r^2-1)-4(-1+r^4)\operatorname{arctanh}(r)\Big)\mp\pi\operatorname{Re}\left(a_{2,2}^\pm\right)r(1-2r^2).\\
	\end{align*}
	Finally, for  $l=3,$
	\begin{align*}
		I^\pm_{0,3}(r)&= -\dfrac{1}{2}\operatorname{Im}\left(a_{0,3}^\pm\displaystyle\int_0^{\pm\pi}\frac{(re^{-{\rm i}\theta}-1)^3}{re^{-{\rm i}\theta}+1} {\rm i}e^{{\rm i}\theta}d\theta\right)\vspace{0.3cm}\\
		&= -\operatorname{Im}(a_{0,3}^\pm)\Big((1+r^2)-8r\operatorname{arctanh}(r)\Big)\mp2\pi \operatorname{Re}(a_{0,3}^\pm) r,\vspace{0.3cm}\\
		I^\pm_{1,3}(r)&= -\dfrac{1}{2}\operatorname{Im}\left(a_{1,3}^\pm\displaystyle\int_0^{\pm\pi}\frac{(re^{{\rm i}\theta}-1)(re^{-{\rm i}\theta}-1)^2 {\rm i}e^{{\rm i}\theta}}{re^{{\rm i}\theta}+1}d\theta\right)\vspace{0.3cm}\\
		&= -\dfrac{\operatorname{Im}\left(a_{1,3}^\pm\right)}{r}\Big(-r(1+r^2)+2(1+r^2)^2\operatorname{arctanh}(r)\Big)\mp \pi\operatorname{Re}\left(a_{1,3}^\pm\right) r^2(1+r^2),\vspace{0.3cm}\\
		I^\pm_{2,3}(r)&= -\dfrac{1}{2}\operatorname{Im}\left(a_{2,3}^\pm\displaystyle\int_0^{\pm\pi}\frac{(re^{-{\rm i}\theta}+1)(re^{{\rm i}\theta}-1)^2(re^{-{\rm i}\theta}-1) {\rm i}e^{{\rm i}\theta}}{(re^{{\rm i}\theta}+1)^2}d\theta\right)\vspace{0.3cm}\\
		&= -\dfrac{\operatorname{Im}\left(a_{2,3}^\pm\right)}{r}\Big(5r(1+r^2)-4(1+r^4)\operatorname{arctanh}(r)\Big)\pm2\pi\operatorname{Re}\left(a_{2,3}^\pm\right) r^4,\\    
		I^\pm_{3,3}(r)&= -\dfrac{1}{2}\operatorname{Im}\left(a_{3,3}^\pm\displaystyle\int_0^{\pm\pi}\frac{(re^{-{\rm i}\theta}+1)^2(re^{{\rm i}\theta}-1)^3 {\rm i}e^{{\rm i}\theta}}{(re^{{\rm i}\theta}+1)^3}d\theta\right)\vspace{0.3cm}\\
		&= -\dfrac{\operatorname{Im}\left(a_{3,3}^\pm\right)}{r}\Big(-5r(1+r^2)+(6-4r^2+6r^4)\operatorname{arctanh}(r)\Big)\mp\pi\operatorname{Re}\left(a_{3,3}^\pm\right) r^2(-1+3r^2).\\
	\end{align*}
	By adding all the above expressions we obtain  $M_1^\pm$ and then the final expression of $M_1.$
\end{proof}

\begin{remark}\label{re:M1} Values of the parameters appearing in the expression of $M_1$ given in Proposition~\ref{M1} when $m\le3.$
	\begin{align*}
		a&=\operatorname{Im}(a_{ 0,0}^+)-\operatorname{Im}(a_{ 0,0}^-)-\operatorname{Im}(a_{ 0,1}^+)+\operatorname{Im}(a_{ 0,1}^-)+\operatorname{Im}(a_{11}^+)-\operatorname{Im}(a_{11}^-)+\operatorname{Im}(a_{0,2}^+)-\operatorname{Im}(a_{0,2}^-)\\&\quad-\operatorname{Im}(a_{1,2}^+)+\operatorname{Im}(a_{1,2}^-)+5\operatorname{Im}(a_{2,2}^+)-5\operatorname{Im}(a_{2,2}^-)-\operatorname{Im}(a_{0,3}^+)+\operatorname{Im}(a_{0,3}^-)+\operatorname{Im}(a_{1,3}^+)\\&\quad-\operatorname{Im}(a_{1,3}^-)-5\operatorname{Im}(a_{2,3}^+)+5\operatorname{Im}(a_{2,3}^-)+5\operatorname{Im}(a_{3,3}^+)-5\operatorname{Im}(a_{3,3}^-),\\     
		b&=-\pi\operatorname{Re}(a_{ 0,0}^+)-\pi\operatorname{Re}(a_{ 0,0}^-)+\pi\operatorname{Re}(a_{11}^+)+\pi\operatorname{Re}(a_{11}^-)+\pi\operatorname{Re}(a_{0,2}^-)+\pi\operatorname{Re}(a_{0,2}^+)-\pi\operatorname{Re}(a_{2,2}^+)\\&\quad-\pi\operatorname{Re}(a_{2,2}^-)-2\pi\operatorname{Re}(a_{0,3}^-)-2\pi\operatorname{Re}(a_{0,3}^+)-\pi\operatorname{Re}(a_{1,3}^+)-\pi\operatorname{Re}(a_{1,3}^-)\\&\quad+\pi\operatorname{Re}(a_{3,3}^+)+\pi\operatorname{Re}(a_{3,3}^-),\\
		c&=-\operatorname{Im}(a_{ 0,0}^+)+\operatorname{Im}(a_{ 0,0}^-)-\operatorname{Im}(a_{ 0,1}^+)+\operatorname{Im}(a_{ 0,1}^-)+\operatorname{Im}(a_{11}^+)-\operatorname{Im}(a_{11}^-)-\operatorname{Im}(a_{0,2}^+)\\&\quad+\operatorname{Im}(a_{0,2}^-)+\operatorname{Im}(a_{1,2}^+)-\operatorname{Im}(a_{1,2}^-)-5\operatorname{Im}(a_{2,2}^+)+5\operatorname{Im}(a_{2,2}^-)+\operatorname{Im}(a_{1,3}^+)-\operatorname{Im}(a_{1,3}^-)\\&\quad-5\operatorname{Im}(a_{2,3}^+)+5\operatorname{Im}(a_{2,3}^-)+5\operatorname{Im}(a_{3,3}^+)-5\operatorname{Im}(a_{3,3}^-),\\
		d&=-\pi\operatorname{Re}(a_{11}^+)-\pi\operatorname{Re}(a_{11}^-)-\pi\operatorname{Re}(a_{1,2}^+)-\pi\operatorname{Re}(a_{1,2}^-)+2\pi\operatorname{Re}(a_{2,2}^+)+2\pi\operatorname{Re}(a_{2,2}^-)\\&\quad-\pi\operatorname{Re}(a_{1,3}^+)-\pi\operatorname{Re}(a_{1,3}^-)+2\pi\operatorname{Re}(a_{2,3}^+)+2\pi\operatorname{Re}(a_{2,3}^-)-3\pi\operatorname{Re}(a_{3,3}^+)-3\pi\operatorname{Re}(a_{3,3}^-),\\
		\alpha&=-2\operatorname{Im}(a_{11}^+)+2\operatorname{Im}(a_{11}^-)-2\operatorname{Im}(a_{1,3}^+)+2\operatorname{Im}(a_{1,3}^-)+4\operatorname{Im}(a_{2,3}^+)-4\operatorname{Im}(a_{2,3}^-)\\&\quad-6\operatorname{Im}(a_{3,3}^+)+6\operatorname{Im}(a_{3,3}^-),\\
		\beta&=-2\operatorname{Im}(a_{1,2}^+)+2\operatorname{Im}(a_{1,2}^-)+4\operatorname{Im}(a_{2,2}^+)-4\operatorname{Im}(a_{2,2}^-),\\
		\gamma&=-\operatorname{Im}(a_{0,3}^+)+\operatorname{Im}(a_{0,3}^-)+\operatorname{Im}(a_{1,3}^+)-\operatorname{Im}(a_{1,3}^-)-\operatorname{Im}(a_{2,3}^+)+\operatorname{Im}(a_{2,3}^-)+\operatorname{Im}(a_{3,3}^+)-\operatorname{Im}(a_{3,3}^-).
	\end{align*}
\end{remark}

In what follows, we establish a connection between the coefficients of the bifurcation functions $M_1$ and $N_1$ of -1 and 1, respectively. This relation allow us to study the simultaneous zeros of these functions.

\begin{proposition}\label{N1}
	Let $M_1(r)= M_1(r; a^\pm_{k,l})$ and $N_1(r)=N_1(r; a^\pm_{k,l})$ be the bifurcation functions of PWCS \eqref{eq_pert} with $f(z)={\rm i}(z^2-1)/2$ associated to $z=-1$ and $z=1,$ respectively. Then
	\[
	N_1(r)=N_1(r; a^\pm_{k,l})= M_1(r; (-1)^la^\mp_{k,l}),
	\]    
	that is, the expression of $N_1(r)$ coincides with the expression of $M_1(r)$ given in Proposition~\ref{M1} and Remark~\ref{re:M1} where each $a_{k,l}^\pm$ is changed  by $(-1)^{l}a_{k,l}^\mp,$ for all $0\leq k\leq l$ and $0\leq l\leq m.$
\end{proposition}
\begin{proof}
	Using the change of variables and time $w(t)=-z(-t), $ we transform PWCS \eqref{eq_pert} with $f(z)={\rm i}(z^2-1)/2$ into 
	\begin{equation}\label{eq_pos}
		\dot{w}={\rm i}(w^2-1)/2+\left\{\begin{array}{l}
			\epsilon R_m^-(-w,-\overline{w}), \text{ when } \operatorname{Im}(w)>0,\\[5pt]
			\epsilon R_m^+(-w,-\overline{w}),\text{ when } \operatorname{Im}(w)<0.
		\end{array} \right.
	\end{equation}
	Hence, the zeros of the bifurcation function $N_1(r)$ of PWCS \eqref{eq_pert} are the zeros of the bifurcation function associated to $z=-1$ of \eqref{eq_pos}. Then 
	\begin{equation}
		R_m^\mp(-w,-\overline{w})=\sum_{l=0}^m\sum_{k=0}^l\overline{a}^\mp_{k,l}(-w)^{l-k}(-\overline{w})^k=\sum_{l=0}^m\sum_{k=0}^l(-1)^l \overline{a}^\mp_{k,l} w^{l-k}\overline{w}^k,
	\end{equation}
	as we wanted to prove.
\end{proof}

From the above proposition we obtain the expression of the function $N_1$ in Proposition~\ref{prop_aux_complex} and also next remark.

\begin{remark}\label{re:N1} Values of the parameters appearing in the expression of $N_1$ given in Proposition~\ref{prop_aux_complex} when $m\le3.$
	\begin{align*}
		\kappa&= 2\pi\big(2\operatorname{Re}(a_{2,2}^-) +2 \operatorname{Re}(a_{2,2}^+) -\operatorname{Re}(a_{1,2}^-) - \operatorname{Re}(a_{1,2}^+) \big),\\
		\rho&= 4\pi\big(\operatorname{Re}(a_{0,3}^+)+\operatorname{Re}(a_{0,3}^-)+\operatorname{Re}(a_{1,3}^+)+\operatorname{Re}(a_{1,3}^-)-\operatorname{Re}(a_{2,3}^+)-\operatorname{Re}(a_{2,3}^-)\\&\quad\quad\quad+\operatorname{Re}(a_{3,3}^+)+\operatorname{Re}(a_{3,3}^-)\big).
	\end{align*}
\end{remark}

\section{Acknowledgements}
This article was possible thanks to the scholarship granted from the Brazilian Federal Agency for Support and Evaluation of Graduate Education (CAPES), in the scope of the Program CAPES-Print, process number 88887.310463/2018-00, International Cooperation Project number 88881.310741/2018-01.

Armengol Gasull is partially supported supported by the Ministry of Science and Innovation--State Research Agency of the Spanish Government through grants PID2022-136613NB-I00 and  by the grant 2021-SGR-00113 from AGAUR of Generalitat de Catalunya.

Gabriel Alexis Rondón Vielma is supported by São Paulo Research Foundation (FAPESP) grants 2020/06708-9 and 2022/12123-9. 

 Paulo Ricardo da Silva is also partially supported by São Paulo Research Foundation (FAPESP) grant 2019/10269-3 and 2023/02959-5, CNPq grant 302154/2022-1 and ANR-23-CE40-0028. 









\bibliographystyle{abbrv}
\bibliography{references1}

\end{document}